\documentclass{amsart}
\usepackage{epsfig}
\usepackage{graphics}
\usepackage{amsfonts,mathrsfs}
\usepackage{amsmath}
\usepackage{latexsym}

\newtheorem{theorem}{Theorem}[section]
\newtheorem{lemma}{Lemma}[section]
\newtheorem{proposition}{Proposition}[section]
\newtheorem{corollary}{Corollary}[section]
\theoremstyle{definition}
\newtheorem{definition}{Proposition-Definition}[section]

\theoremstyle{remark}
\newtheorem{remark}{Remark}[section]
\numberwithin{equation}{section}

\DeclareMathOperator{\supp}{supp} 

\begin{document}
%

\newcommand\gs{\mathcal{T},\sigma}
\newcommand\nugs{\nu_{\mathcal{T},\sigma}}
\newcommand\nur{\nu_{\rho}}
\newcommand\nurgs{\nu_{\rho,\mathcal{T},\si}}
\newcommand\jr{j_{\rho}}
\newcommand\ijkr{I_{j,k,\rho}}
\newcommand\iprr{I_{p,r,\rho}}
\newcommand\ijkrd{[kb^{-j}, kb^{-j}+b^{-j\rho})}
\newcommand\hnu{h_\nu}
\newcommand\hnur{h_{\nu_\rho}}
\newcommand\hnugs{h_{\nu_{\mathcal{T},\sigma}}}
\newcommand\hnurgs{h_{\nurgs}}
\newcommand\ehnu{E_h^\nu}
\newcommand\ehnur{E_h^{\nur}}
\newcommand\ehnugs{E_h^{\nugs}}
\newcommand\ehnurgs{E_h^{\nurgs}}
\newcommand\fhgs{F_{h,\mathcal{T},\si}}
\newcommand\fhpgs{F_{h',\mathcal{T},\si}}
\newcommand\fhrgs{F_{h,\rho,\mathcal{T},\si}}
\newcommand\fhprgs{F_{h',\rho,\mathcal{T},\si}}
\newcommand\almu{\alpha_{\mu}}
\newcommand\ealmu{E_\alpha^{\mu}}
\newcommand\ealmub{\overline{E}_\alpha^{\mu}}
\newcommand\ealmud{E_{\alpha,\delta}^{\mu}}
\def\ga{\mathcal{T}}
\def\si{\sigma}
\def\dnu{d_\nu}
\def\dmub{\overline{d}_\mu}
\def\dnur{d_{\nur}}
\def\dnugs{d_{\nugs}}
\def\dnurgs{d_{\nurgs}}
\def\dmu{d_\mu}
\def\dmut{\tilde{d_\mu}}
\newcommand\ealmut{\widetilde{E}_\alpha^{\mu}}
\newcommand\ealmutd{\widetilde{E}_{\alpha,\delta}^{\mu}}
\newcommand\almum{\underline{\alpha}_{\mu}}
\newcommand\almuu{\overline{\alpha}_{\mu}}
\newcommand\fald{\widetilde{F}_{\alpha,\delta}^{\mu}}
\def\2mj{b^{-j}} 
\def\mual{\mu_\alpha}
\def\muald{\mu_{\alpha,\delta}}
\def\hald{h_{\alpha,\delta}}
\newcommand\alp{h_f}
\def\sald{S_{ \alpha,\delta}}
\def\salbdr{S_{\rho,\alpha,\beta,\delta}}
\def\salepdri{S_{\rho,\alpha_i,\alpha_i+\ep,\delta_i}}
\def\pald{{\mathcal P}(\alpha,\delta,\ep)}
\def\paldr{{\mathcal P}(\rho,\alpha,\delta,\ep)}
\def\djk{d_{j,k}}
\def\psijk{\psi_{j,k}}
\def\bproof{\begin{proof}}
\def\eproof{\end{proof}}
\def\ra{\rightarrow}
\def\ma{\mapsto}
\def\un{{\relax\ifmmode 1\!\!1\else$1\!\!1$\fi}}
\newcommand\ho{H\"{o}lder }
\def\alphamin{\alpha_{\min}}
\def\alphamax{\alpha_{\max}}
\def\almu{\alpha_\mu}
\def\almup{\alpha^+_\mu}
\def\almum{\alpha^-_\mu}
\def\lag{\langle}                   
\def\rag{\rangle}                   
\def\ep{\varepsilon}
\def\phijk{\phi_{j,k}}
\def\ijk{I_{j,k}}
\def\bs{\backslash}
\def\equi{\Leftrightarrow}                

\def\pqx{p_{q,x}}
\def\iqx{I_{q}(x)}
\def\iqxp{I_{q}^{+}(x)}
\def\iqxm{I_{q}^{-}(x)}
\def\ipqp{I_{p,q}^{+}}
\def\ipqr{I_{p,q,\rho}}
\def\ipqm{I_{p,q}^{-} }
\def\ipq{I_{p,q}}
\def\iqxp{I_{q}^{+}(x)}
\def\iqxm{I_{q}^{-}(x)}
\def\pqx{p_{q,x}}
\def\pqx{p_{q,x}}
\def\pqxp{p^+_{q,x}}
\def\pqxm{p^-_{q,x}}

\def\ijkb{I^b_{j,k}}
\def\ijkc{I^c_{j,{\bf k}}}
\def\wijkc{\widetilde{I}^c_{j,k}}
\def\hijkc{\widehat{I}^c_{j,k}}
\def\ijx{I_{j}(x)}
\def\ijxp{I_{j}^{+}(x)}
\def\ijxm{I_{j}^{-}(x)}
\def\ijxb{I_{j}^{b}(x)}
\def\ijxc{I_{j}^{c}(x)}
\def\ijyc{I_{j}^{c}(y)}
\def\ijxcd{I_{j}^{c \,(\delta)}(x)}
\def\kjx{k_{j,x}}
\def\kjxb{{\bf k}^b_{j,x}}
\def\kjxc{{\bf k}^c_{j,x}}
\def\kjyc{{\bf k}^c_{j,y}}
\def\kjxp{k^+_{j,x}}
\def\kjxm{k^-_{j,x}}







\title[Birkhoff averages  on "self-affine" symbolic spaces]{Multifractal analysis of Birkhoff averages on "self-affine" symbolic spaces}
\author{Julien Barral}

\address{INRIA Rocquencourt, B.P. 105, 78153 Le
Chesnay Cedex, France}
\email{julien.barral@inria.fr}
\author{Mounir Mensi}
\address{Facult\'e des Sciences de Monastir, 5019 Monastir, Tunisia}
\email{mounirmensi@yahoo.fr}

\begin{abstract}
We achieve on self-affine Sierpinski carpets the multifractal analysis of the Birkhoff averages of potentials satisfying a Dini condition.  Given such a potential, the corresponding Hausdorff spectrum cannot be deduced from that of the associated Gibbs measure by a simple transformation. Indeed, these spectra are respectively obtained as the Legendre transform of two distinct concave differentiable functions that cannot be deduced from one another by a dilation and a translation. This situation is in contrast with what is observed in the familiar self-similar case. Our results are presented in the framework of almost-multiplicative functions on products of two distinct symbolic spaces and their projection on the associated self-affine carpets.
\end{abstract}
\maketitle
\section{Introduction}

Let $\varphi$ be a real valued continuous potential defined on the one sided symbolic space $\mathbb{A}$ over a finite alphabet $A$ of cardinality $\# A\ge 2$, endowed with the shift operation denoted $\sigma$ and the standard metric. As usual the Birkhoff sums associated with $\varphi$ are defined by $S_n\varphi =\sum_{k=0}^{n-1} \varphi \circ \sigma^k $. The Hausdorff spectrum associated with the asymptotic behavior of the Birkhoff averages of $\varphi$ is the mapping ($\dim$ stands for the Hausdorff dimension and we refer to \cite{Falcbook} or \cite{Pesin} for its definition)
\begin{equation}\label{spec1}
\alpha\in\mathbb{R} \mapsto \dim\, \left\{x\in \mathbb{A}:\ \lim_{n\to\infty}\frac{S_n\varphi(x)}{n}=\alpha\right\}.\end{equation}
This mapping, as well as its analogue when the symbolic space $\mathbb{A}$ is replaced by its geometric realization on a conformal repeller has been studied intensively \cite{Collet,Rand,BMP,PeW,Fan,O,FengLauWu,Ke,FengO} (extensions replacing $S_n\varphi(x)$ by vector valued Birkhoff sums or the logarithm of the norm of Birkhoff products of positive matrices have also been considered \cite{Fan0,Fan1,F}). In all these situations, to finding the Hausdorff spectrum of the Birkhoff averages of $\varphi$ amounts to finding the singularity spectrum of the (possibly weak,  see \cite{Ke,FengO}) Gibbs measure $\mu$ associated with $\varphi$, i.e. the mapping
\begin{equation}\label{spec2}
\alpha\in\mathbb{R}_+ \mapsto \dim\, \left\{x\in \mathbb{A}:\ \lim_{r\to 0^+}\frac{\log \mu(B(x,r))}{\log (r)}=\alpha\right\}.\end{equation}
Indeed, the spectra (\ref{spec1}) and (\ref{spec2}) are deduced from one another by a dilation and a translation. Moreover, the measure $\mu$ obeys the multifractal formalism in the sense that its singularity spectrum is obtained as the Legendre transform $\tau_\mu^*(\alpha)=\inf_{q\in\mathbb{R}}\alpha q-\tau_\mu(q)$ of the $L^q$-spectrum of $\mu$, which is defined by 
$$
\tau_\mu(q)=\liminf_{r\to 0}\ \frac{\log \sup \left\{\sum_i \mu(B_i)^q\right \}}{\log (r)},
$$
where the supremum is taken over all the  families of disjoint closed balls $B_i$ of radius equal to $r$ with centre in $\supp(\mu)$ (see \cite{BMP,Be,Olsen,Pesin,Pe} for detailed expositions of multifractal formalisms).  Besides, $\tau_\mu(q)$ is related to the topological pressure $P(q\varphi)$ of $q\varphi$ by the relation $\tau_\mu(q)=\big (qP(\varphi)-P(q\varphi)\big )/\log (\#A)$. 

The same kind of relation holds between the Hausdorff spectra of Birkhoff averages and Gibbs measures on conformal repellers. 

Symbolic spaces and conformal repellers share the property  to be endowed with a metric distance giving them a self-similar structure, which is responsible for a so simple relation. It is thus natural to investigate the relationship between the Hausdorff spectra of Birkhoff averages and Gibbs measures when this self-similarity property is relaxed. The most natural way to do this consists in endowing the symbolic  space  with a metric distance making its geometry very close to that of a self-affine Sierpinski carpet, in the sense that balls are naturally projected on "almost squares" in the plane, while cylinders are projected on the rectangular cells covering the carpet. 

The multifractal analysis of Gibbs measures on such "self-affine" symbolic spaces (see Section~\ref{self-affine} for the definition) and their projections on self-affine Sierpinski carpets has been achieved in \cite{King,Olsen} for the special case where the potential is constant over the cylinders of the first generation, and in \cite{barmen} for general potentials satisfying a Dini condition. It turns out that the singularity spectrum of the Gibbs measure $\mu$ associated with such a potential $\varphi$ is given by the Legendre transform of a differentiable function $\beta_\varphi$. In general the function $\beta_\varphi$ differs from the $L^q$-spectrum. This difference with respect to the self-similar case recalls and extends a phenomenon arising in the study of the dimensions of self-affine Sierpinski carpets and more general self-affine sets \cite{McMullen,Bedford,PW,KePe,Feng}:  The Hausdorff and box dimensions may be different (see also \cite{Falc1,Falc2,Falc3} for works related to this issue).  

In this paper we complete the work achieved in \cite{barmen}. Given a Dini potential  $\varphi$ on a self-affine symbolic space or on its projection to a self-affine Sierpinski carpet, we perform the multifractal analysis of the Birkhoff averages of $\varphi$. Our results show that in absence of self-similarity, the associated Hausdorff spectra are obtained as the Legendre transform of a differentiable function $\mathcal{T}_\varphi(q)$, which cannot be simply related to $\beta_\varphi(q)$, and like $\beta_\varphi(q)$ cannot be related linearly to the topological pressure of $q\varphi$. Nevertheless, it is worth noting that the differentiability of $\mathcal{T}_\varphi$ is obtained as a consequence of that of $\beta_\varphi$ (see Proposition~\ref{prop4}(2)).

Actually, we adopt a framework  in which Birkhoff sums of potentials appear as special examples among more general objects. Specifically, we achieve on self-affine symbolic spaces the multifractal analysis of the asymptotic behavior of {\it almost multiplicative functions of cylinders} (see Section~\ref{definition} for the definition, properties and examples). Our result, namely Theorem~\ref{th2} (Section~\ref{Th2}), requires to extend the class of Gibbs measures considered in \cite{barmen}. The singularity spectrum of such a measure is provided by Theorem~\ref{th1} (Section~\ref{Th1}). For this result, the ideas of the proof are the same as for the special case studied in \cite{barmen}, but we write in detail some arguments whose verifications, though not immediate, were left to the reader in \cite{barmen}. Theorems~\ref{th2} and~\ref{th1} are proved simultaneously in Section~\ref{proof1}. Our results regarding the multifractal analysis of Birkhoff averages and Gibbs measures on self-affine Sierpinski carpets are respectively described in Theorems~\ref{proj1} and~\ref{proj2} in Section~\ref{proj}. Theorems~\ref{proj1} appears to be a direct consequence of Theorem~\ref{th2}. The multifractal analysis of Gibbs measures on Sierpinski carpets requires to complete the work achieved on the symbolic space. The approach is the same as the approach explained in detail in \cite{barmen} for a smaller class of Gibbs measures, and we will not provide a proof of Theorem~\ref{proj2}.

\section{Almost-multiplicative functions and "self-affine" symbolic spaces}\label{definition}

Let $A$ be a finite set of cardinality $r\ge 2$, $A^*=\bigcup_{k\ge 0}A^k$  the set of finite words built over the alphabet $A$ (with the convention $A^0=\{\emptyset\}$), and $\mathbb{A}=A^{\mathbb{N^*}}$, the symbolic space  over $A$.  

The set $A^*\cup\mathbb{A}$ is endowed with the concatenation operation. If $w\in A^*$, $[w]$ denotes the cylinder $w\cdot \mathbb{A}=\{w\cdot w',\ w'\in\mathbb{A}\}$. The set of cylinders is denoted by $\mathscr{C}_A$ and it is naturally endowed with the concatenation operation $[w]\cdot [w']=[w\cdot w']$. 

The shift operation on $\mathbb{A}$ is denoted by $\sigma$. 

If $z=z_1z_2\cdots  z_p\cdots \in \mathbb{A}$ and $n\in\mathbb{N}$ then $z|n$ stands for the prefix $z_1\cdots z_n$ of $z$ if $n\ge 1$ and the empty word otherwise.

The length of a word $w\in A^*\cup\mathbb{A}$ is denoted by $|w|$. 

For $z,z'\in A^*\cup\mathbb{A}$, let $z\land z'$ stands for the word $u$ of maximal length in $A^*\cup\mathbb{A}$ such that $u$ is a prefix of $z$ and $z'$.  

The set $\mathbb{A}$ is endowed with the ultrametric  distance $d:(z,z')\in \mathbb{A}^2\mapsto r^{-|z\land z'|}$.

We denote by $\lambda$ be the unique measure on $\mathbb{A}$ such that $\lambda([w])=r^{-n}$ for all $n\ge 0$ and  $w\in A^n$. 

If $\{f_i\}_{i\in I}$ and $\{g_i\}_{i\in I}$ are two families of real-valued functions defined on a set $E$, when we say that there exists $C>0$ such that $f_i(x)\approx g_i(x)$ for all $i\in I$ and $x\in E$, that means that there exists $C>0$ such that $C^{-1}f_i(x)\le g_i(x)\le C f_i(x)$ for all $i\in I$ and $x\in E$. When there is no ambiguity, we just write $f_i(x)\approx g_i(x)$ for all $i\in I$ and $x\in E$ to mean that such a constant exists. This notation enables to shorten certain mathematical formulas.

\subsection{\bf Almost-multiplicative functions and quasi-Bernoulli measures}\label{AM}

\begin{definition}\label{AM1} We say that a non-zero nonnegative function $\psi: \mathscr{C}_A\to \mathbb{R}_+$ is  almost-multiplicative if there exists $C>0$ such that for all $([w],[\widetilde w])\in \mathscr{C}_A^2$ we have
\begin{equation}\label{bernoulli}
C^{-1}\psi ([w])\psi([\widetilde w])\le \psi ([w]\cdot[w'])\le C \psi ([w])\psi([\widetilde w]).
\end{equation}
The support of $\psi$ is defined by 
$$
\mbox{supp}(\psi)=\left\{z\in\mathbb{A}: \forall \ n\ge 0, \psi ([z|n])>0\right \}.$$
By construction, if $[A]$ stands for the set $\{[a]: a\in A,\ \psi([a])>0\}$, then 
$$
\supp(\psi)=\bigcap_{n\ge 0}\sigma^{-n}([A]).
$$

Due to (\ref{bernoulli}) the sequence $\big(s_n=\sum_{w\in A^n}\psi([w])\big )_{n\ge 1}$ is submultiplicative in the sense that $C^{-1}s_ns_p\le s_{n+p}\le  C s_n s_p$ for all $n,p\ge 1$. Consequently, the topological pressure $P(\log \psi)$ of the function $\log \psi$ is well defined by 
\begin{equation}\label{entr}
P(\log \psi)=\lim_{n\to\infty}\frac{1}{n}\log \sum_{w\in A^n}\psi([w])
\end{equation}
and for all $n\ge 1$ we have 
\begin{equation}\label{ent}
C^{-1}\exp (nP(\log \psi))\le\sum_{w\in A^n}\psi([w])\le C\exp (nP(\log \psi)).
\end{equation}

We denote by $AM(\mathscr{C}_A)$ the set of almost-multiplicative functions on $\mathscr{C}_A$.
\end{definition}
For the sake of completeness, we sketch the proofs of (\ref{entr}) and (\ref{ent}).
\begin{proof}
For $n\ge 1$ set $u_n=\log (s_n)+\log (C)$ and $v_n=-\log(s_n)+\log (C)$. The sequence $(u_n)_{n\ge 1}$ is subadditive so $u_n/n$ converges  as $n\to\infty$ to a limit that we denote $P(\log(\psi))$. Also, $P(\log(\psi))\le u_n/n$ for all $n\ge 1$. Similarly, $-P(\log(\psi))\le v_n/n$ for all $n\ge 1$. These two inequalities yield (\ref{ent}).
\end{proof}

If $\psi\in AM(\mathscr{C}_A)$ is the restriction to $\mathscr{C}_A$ of a probability measure, the "almost multiplicativity" property is nothing but the quasi-Bernoulli property introduced in \cite{Michon,Michon2,BMP}. Examples of such functions are given in Section \ref{Examples}. 

Since for  $\psi\in AM(\mathscr{C}_A)$ our purpose will be to study the asymptotic behavior of $\log \psi([w])/|w|$, it will be possible to assume, without loss of generality, that $\psi$ is the restriction to $\mathscr{C}_A$ of a quasi-Bernoulli measure. Let us explain why it is so. It follows from (\ref{bernoulli}) and (\ref{ent}) that if we define 
$$
\widetilde \psi ([w])=\frac{\psi([w])}{\sum_{w'\in A^n}\psi([w'])},\ [w]\in \mathscr{C}_A,
$$
then $\widetilde\psi\in AM(\mathscr{C}_A)$. Consider  on $\mathbb{A}$ the sequence of probability measures 
$$
\mu_n=r^n\sum_{w\in A^n}\widetilde \psi ([w])\lambda_{|[w]},\quad n\ge 1,
$$
where $\lambda_{|[w]}$ stands for the restriction of $\lambda$ to $[w]$. Due to (\ref{bernoulli}) and (\ref{ent}), there exists $C_1>0$ such that for all $n\ge 1$, $p\ge n$ and $ w\in A^n$ 
$$
C_1^{-1}\widetilde \psi ([w])\le \mu_p([w])\le C_1\widetilde \psi([w]).
$$ Consequently, for any measure $\mu$ which is the limit of a subsequence of $(\mu_n)_{n\ge 1}$, for all $n\ge 1$ and $w\in A^n$ we have
$$
C_1^{-1}\widetilde \psi ([w])\le \mu([w])\le C_1\widetilde \psi([w]),
$$
and thus for some other constant $C_2>0$
\begin{equation}\label{Gibbs}
C_2^{-1}\exp(-nP(\log \psi)) \psi([w]) \le \mu([w])\le C_2\exp(-nP(\log \psi)) \psi([w]) .
\end{equation}
In particular, the restriction of $\mu$ to $\mathscr{C}_A$ belongs to $AM(\mathscr{C}_A)$ and $\supp(\mu)=\supp(\psi)$. This implies that there exists on $(\supp(\psi),\sigma)$ an unique ergodic measure $\mu_\psi$ strongly equivalent to the restriction of each of these measures $\mu$ to $\supp(\psi)$ (see \cite{Carleson,He}). This measure deserves to be called Gibbs measure since (\ref{Gibbs}) is an extension of the relation linking the Birkhoff sums and the Gibbs measure associated with a potential satisfying the Dini condition (\ref{Dini}) in the thermodynamic formalism (see \cite{Bowen} and Section~\ref{ex1}). 

We also denote by $\mu_\psi$ the extension of $\mu_\psi$ to $\mathbb{A}$ defined by $\mu_\psi(E)=\mu_\psi(E\cap \supp(\psi))$ for any Borel subset $E$ of $\mathbb{A}$.

\subsection{\bf "Self-affine" symbolic spaces}\label{self-affine}

Let  $2\le r_1\le r_2$  be two integers. For $i\in\{1,2\}$ let $A_i$ denote the set $\{0,\dots,r_i-1\}$.  The symbolic spaces built from $A_1$ and $A_2$ like $\mathbb{A}$ is from $A$ above are respectively denoted by $\mathbb{A}_1$ and $\mathbb{A}_2$. The symbolic space built from $A_1\times A_2$ is equal to $\mathbb{A}_1\times \mathbb{A}_2$. The shift operations over the symbolic spaces $\mathbb{A}_1$, $\mathbb{A}_2$ and $\mathbb{A}_1\times\mathbb{A}_2$ are respectively denoted by $\sigma_1$, $\sigma_2$ and $\sigma$ (we have $\sigma=(\sigma_1,\sigma_2)$). 

\smallskip

Each set $\mathbb{A}_i$, $i\in\{1,2\}$ is endowed with the ultrametric  distance $d_i:(z,z')\in \mathbb{A}_i^2\mapsto r_i^{-|z\land z'|}$. Then, the product symbolic space $\mathbb{A}_1\times \mathbb{A}_2$ is endowed with the ultrametric distance $d\big ((x,y),(x',y')\big )=\max \big (d_1(x,x'),d_2(y,y')\big )$ and called self-affine because the natural projections of its cylinders in the euclidean plane are the rectangular cells of the self-affine grid $\big\{[kr_1^{-n},(k+1)r_1^{-n}]\times [lr_2^{-n},(l+1)r)2^{-n}]\big \}_{\substack{n\ge 1\\ 0\le k<r_1^n,\, 0\le l<r_2^n}}$.

For every $n\ge 1$, let $\mathcal {F}_n$ be the set of balls of radius $r_2^{-n}$ in $(\mathbb{A}_1\times\mathbb{A}_2,d)$. Let $g(n)$ be the smallest integer $m$ such that $r_1^{-m}\le r_2^{-n}$. We have  
\begin{equation}\label{Fn}
\mathcal {F}_n=\left \{[w_1\cdot\widetilde w_1]\times [w_2]: (w_1,\widetilde w_1,w_2)\in A_1^n\times A_1^{g(n)-n}\times A_2^n\right \}.
\end{equation}

If $z\in \mathbb{A}_{1}\times\mathbb{A}_{2}$ and $n\ge 0$, let $B_n(z)$ stand for the unique element of $\mathcal{F}_n$ containing $z$. We have $B_n(z)=B(z,r_2^{-n})$.

The diameter of a subset $E$ of $ (\mathbb{A}_{1}\times\mathbb{A}_{2},d)$ is denoted by $|E|$.

For $i\in\{1,2\}$, let $\pi_i$ stand for the projection from $\mathbb{A}_1\times\mathbb{A}_2$ to $\mathbb{A}_i$.

\subsection{\bf Some properties of almost multiplicative functions} \label{properties}
From now on, we adopt the convention $0^\gamma=0$ for all $\gamma\in\mathbb{R}$.  

\smallskip

If $\psi\in AM(\mathscr{C}_{A_1\times A_2})$ and $q\in\mathbb{R}$, we define 

$$
I_{\psi,q}: [w_1]\in \mathscr{C}_{A_1}\mapsto \sum_{w_2\in A_2^{|w_1|}} \psi([w_1]\times [w_2])^q.
$$

\begin{proposition}\label{AMcont1}
Let $(\psi, \widetilde\psi)\in AM(\mathscr{C}_{A_1\times A_2})^2$. 

\begin{enumerate} 

\item For every $(q,\gamma)\in\mathbb{R}^2$, the function $I_{\psi,q}^\gamma$ belongs to $AM(\mathscr{C}_{A_1})$ and its support is $\pi_1(\supp(\psi))$. 

\item If $(\gamma,\widetilde\gamma)\in\mathbb{R}^2$ and $q\in\mathbb{R}$ the function $I_{\psi,1}^{q\gamma}\cdot I_{\widetilde \psi, q}^{\widetilde \gamma}$ belongs to $AM(\mathscr{C}_{A_1})$. Also, the mapping $
q\in\mathbb{R}\mapsto P\big (I_{\psi,1}^{q\gamma }\cdot I_{\widetilde \psi, q}^{\widetilde \gamma}\big )=P\big (q\gamma \log (I_{\psi, 1})+\widetilde \gamma \log (I_{\widetilde\psi,q})\big )$
is convex. 
\end{enumerate}
\end{proposition}
\begin{proof}
(1) It is enough to prove the result for $\gamma=1$. Indeed, it is immediate that if $\rho\in AM(\mathscr{C}_{A_1})$ and $\gamma\in\mathbb{R}$ then $\rho^\gamma\in AM(\mathscr{C}_{A_1})$ and $\supp(\rho^\gamma)=\supp(\rho)$. Let $q\in\mathbb{R}$. Due to (\ref{bernoulli}), there exists $C>0$ depending on $\psi$ and $q$ such that for all $n,p\ge 1$ and $(w_1,\widetilde w_1)\in A_1^n\times A_1^p$ we have 
\begin{eqnarray*}
I_{\psi,q}([w_1\cdot\widetilde w_1])&=&\sum_{(w_2,\widetilde w_2)\in A_2^n\times A_2^p} \psi([w_1\cdot\widetilde w_1]\times [w_2\cdot\widetilde w_2]^q\\
&\approx& \sum_{(w_2,\widetilde w_2)\in A_2^n\times A_2^p} \psi([w_1]\times[w_2])^q \psi ([\widetilde w_1]\times [\widetilde w_2]^q \\
&\approx&  \sum_{w_2\in A_2^n} \psi([w_1]\times[w_2])^q\sum_{\widetilde w_2\in\times A_2^p} \psi ([\widetilde w_1]\times [\widetilde w_2]^q\\
&=&I_{\psi,q}([w_1])I_{\psi,q}([\widetilde w_1]).
\end{eqnarray*}

Let $x\in \supp(I_{\psi,q})$. By definition, $I_{\psi,q}([x|n])>0$ for all $n\ge 1$. In particular, $I_{\psi,q}([x|1])>0$ and we can chose $y_1\in A_2$ such that $\psi(x|1,y_1)>0$. Suppose that for $n\ge 1$ we have built $(y_1,\dots y_n)\in A_2^n$ such that $\psi (x|n,y_1\cdots y_n)>0$. Since $I_{\psi,q}\in AM(\mathscr{C}_{A_1})$, we have $I_{\psi,q}([x_{n+1}])>0$ so there exists $y_{n+1}\in A_2$ such that $\psi ([x_{n+1}]\times [y_{n+1}])>0$ and thus $\psi([x|n+1,y_1\cdots y_{n+1})])>0$. Consequently, we can construct by induction an infinite word $y\in \mathbb{A}_2$ such that $(x,y)\in K$, hence $\supp(I_{\psi,q})\subset \pi_i(K)$. The opposite inclusion is immediate. 

\smallskip

\noindent
(2) Let $(\gamma,\widetilde\gamma)\in\mathbb{R}^2$. It follows from (1) that for any $q\in\mathbb{R}$, $I_{\psi}^{q\gamma}\cdot I_{\widetilde \psi, q}^{\widetilde \gamma}$ belongs to $AM(\mathscr{C}_{A_1})$ as a product of two elements of $AM(\mathscr{C}_{A_1})$.

Fix $[w_1]\in \mathscr{C}_{A_1}$. On the one hand, the mapping $q\mapsto I_{\widetilde \psi,q}([w_1])$ is log-convex as the sum of the log-convex functions $q\mapsto \widetilde \psi([w_1]\times[w_2])^q$, $w_2\in A_2^{|w_1|}$, so $q\mapsto I_{\widetilde \psi,q}([w_1])^{\widetilde \gamma}$ is log-convex as well. On the other hand, the mapping $q\mapsto I_{\psi,1}([w_1])^{q \gamma}$ is clearly log-convex. Consequently, $q\mapsto I_{\psi,1}([w_1])^{q \gamma}\cdot I_{\widetilde \psi,q}([w_1])^{\widetilde \gamma}$ is log-convex. It follows from these remarks and the definition of the pressure of an element of $AM(\mathscr{C}_{A_1})$ (see (\ref{entr})) that $q\in\mathbb{R}\mapsto P\big (I_{\psi,1}^{\gamma q}\cdot I_{\widetilde \psi, q}^{\widetilde \gamma}\big )$ is the limit of a sequence of convex functions. 
\end{proof}

\begin{proposition}\label{AMcont}
Let $\psi\in AM(\mathscr{C}_{A_1\times A_2})$. Suppose that $\psi$ is the restriction to $\mathscr{C}_{A_1\times A_2}$  of a Gibbs measure $\mu_\psi$. Then 
\begin{enumerate}
\item $\mu_\psi\circ \pi_1^{-1}\in AM(\mathscr{C}_{A_1})$ and $\supp(\mu_\psi\circ\pi_1^{-1})=\pi_1(\supp(\psi))$.

\item  There exists $C>0$ such that 
$$
\mu_\psi \big ([w_1\cdot\widetilde w_1]\times [w_2]\big )\approx \psi \big
([w_1]\times [w_2]\big )\mu_\psi\circ \pi_1^{-1}([\widetilde w_1]),
$$
for all $n\ge 1$ and ball $B=[w_1\cdot\widetilde w_1]\times [w_2]\in \mathcal{F}_n$.
\end{enumerate}
\end{proposition}
\begin{proof}
(1) It is enough to notice that the restriction of $\mu_\psi\circ\pi_1^{-1}$ to $AM(\mathscr{C}_{A_1})$ is equal to the function $I_{\psi,1}$. 

\smallskip

\noindent (2) For $n\ge 1$ and $B=[w_1\cdot\widetilde w_1]\times [w_2]\in \mathcal{F}_n$ we have 
\begin{eqnarray*}
\mu_\psi \big ([w_1\cdot\widetilde w_1]\times [w_2]\big )&=&
\sum_{\widetilde w_2\in A_2^n}\mu_\psi \big ([w_1\cdot\widetilde w_1]\times [w_2\cdot\widetilde w_2]\big )\\
&\approx &  \sum_{\widetilde w_2\in A_2^n}\psi([w_1]\times[w_2])\psi([\widetilde w_1]\times[\widetilde w_2])\\
&=& \psi([w_1]\times[w_2])\sum_{\widetilde w_2\in A_2^n}\psi([\widetilde w_1]\times[\widetilde w_2])\\
&=&\psi([w_1]\times[w_2])\mu_\psi\circ \pi_1^{-1}([\widetilde w_1]).
\end{eqnarray*}
\end{proof}

\subsection{{\bf Examples}}\label{Examples}
Let us fix a subset $A\subset A_1\times A_2$ such that $\# A\ge 2$ (to avoid trivial cases). Then let $[A]=\{[a_1]\times[a_2]: (a_1,a_2)\in A\}$ and consider the invariant compact subset of $\mathbb{A}_1\times\mathbb{A}_2$ defined by
$$
K=\bigcap_{n\ge 0}\sigma^{-n}([A]).
$$
The set $K$ will be the support of the functions $\psi$ that we are going to consider. The first example is a special case of the second one, which is itself a particular case of the third one.  

\subsubsection{{\bf $\psi$ as the exponential of Birkhoff sums.}\label{ex1}}

Let $\varphi: K\to \mathbb{R}$ be a continuous function satisfying the Dini condition 
\begin{equation}\label{Dini}
\int_{[0,1]}\sup_{\substack{z,z'\in K\\ d(z,z')\le r}}|\varphi(z)-\varphi(z')|\ \frac{dr}{r}<\infty
\end{equation}
(this holds for instance if $\varphi$ is H\"older continuous). Equivalently, 
\begin{equation}\label{bd}
Var(\varphi):=\sup_{n\ge 0}\sup_{(w_1,w_2)\in A_1^n\times A_2^n}\sup_{z,z'\in K\cap [w_1]\times [w_2]} |S_n\varphi(z')-S_n\varphi(z)| <\infty,
\end{equation}
with the convention, adopted also in the sequel, that the supremum over the empty set is equal to 0. Then if for $n\ge 1$ and $(w_1,w_2)\in A_1^n\times A_2^n$ we define
$$
\psi([w_1]\times [w_2])=\sup_{z\in K\cap [w_1]\times [w_2]}\exp \big (S_n\varphi (z)\big ),
$$
property (\ref{bd}) implies that $\psi\in AM(\mathscr{C}_{A_1\times A_2})$. This is the well known principle of bounded distortions.
\subsubsection{{\bf $\psi$ as the norm of Birkhoff products of matrices.}}\label{ex2}
Fix an integer $d\ge 1$ and for $1\le i,j \le d$ let $\varphi_{i,j}: K\to \mathbb{R}$ satisfying the Dini condition (\ref{Dini}). Then let $M$ be the mapping from $K$ to $\mathcal{M}_d(\mathbb{R}_+^*)$ , the set of square matrices of order $d$ whose entries are positive, defined by $M_{i,j}=\exp (\varphi_{i,j})$ for $1\le i,j\le d$. Finally, for $n\ge 1$ and $(w_1,w_2)\in A_1^n\times A_2^n$ let 
$$
\psi([w_1]\times [w_2])=\sup_{z\in K\cap [w_1]\times [w_2]}\|M\big (\sigma^{n-1}z\big )\cdots M(z)\|.
$$
Here $\|\cdot\|$ denotes the matrix norm defined by $\|B\|:=\mathbf{1}^\tau B\mathbf{1}$, where $\mathbf{1}$ is the $d$-dimensional column vector each coordinate of which is 1. 

Let us check that $\psi\in AM(\mathscr{C}_{A_1\times A_2})$. For $z\in K$ we denote $M\big (\sigma^{n-1}z\big )\cdots M(z)$ by $\xi_nM(z)$. The submultiplicativity property of $\|\cdot\|$ implies $\|\xi_{n+p}M(z)\|\le\|\xi_nM(z)\|\cdot\|\xi_pM(\sigma^n z)\|$ for all $n,p\ge 1$. This yields $\psi([w_1\cdot\widetilde w_1]\times [w_2\cdot\widetilde w_2])\le \psi([w_1]\times[w_2])\psi([\widetilde w_1]\times[\widetilde w_2])$ for all pairs of cylinders $([w_1]\times[w_2],[\widetilde w_1]\times[\widetilde w_2])$ in $\mathscr{C}_{A_1\times A_2}^2$. 

We now deal with the opposite inequality. Since $M$ is continuous, Lemma 2.1 in  \cite{FNonlinearity} yields $C>0$ such that for all $n,p\ge 1$ and $z\in K$,
\begin{equation}\label{fg}
\|\xi_{n+p}M(z)\|\ge C \|\xi_nM(z)\|\cdot \|\xi_pM(\sigma^n z)\|.
\end{equation}
Now fix $z\in K$, $n\ge 1$ and $z'\in K\cap [z|n]$. We have 
$$
\|\xi_nM(z')\|=\sum_{\substack{1\le i\le d\\1\le j\le d}}\sum_{\substack{1\le k_0,\dots, k_{n}\le d\\ k_0=i,\ k_n=j}}\prod_{l=0}^{n-1}M_{k_l,k_{l+1}}(\sigma^lz').
$$
Due to the Dini property satisfied by the entries of $M$, by using (\ref{bd}) we get
$$
\displaystyle \prod_{l=0}^{n-1}M_{k_l,k_{l+1}}(\sigma^lz')\ge \widetilde C^{-1} \prod_{l=0}^{n-1}M_{k_l,k_{l+1}}(\sigma^lz),
$$
where $\displaystyle \widetilde C=\prod_{\substack{1\le i\le d\\1\le j\le d}}\exp (Var(\varphi_{i,j}))$. Consequently, due to (\ref{fg}), for all $n,p\ge 1$ and $z\in K$ we have 
$$
\|\xi_{n+p}M(z)\|\ge C\widetilde C^{-2} \sup_{z'\in K\cap [z|n]} \|\xi_nM(z)\|\, \sup_{z'\in K\cap [\sigma^n z|p]}  \|\xi_pM(\sigma^n z')\|.
$$
This yields $\psi([w_1\cdot\widetilde w_1]\times [w_2\cdot\widetilde w_2])\ge C\widetilde C^{-2}  \psi([w_1]\times[w_2])\psi([\widetilde w_1]\times[\widetilde w_2])$ for all pairs of cylinders $([w_1]\times[w_2],[\widetilde w_1]\times[\widetilde w_2])\in \mathscr{C}_{A_1\times A_2}^2$.

\subsubsection{{\bf $\psi$ as a kind of skew product.}}\label{ex3}

Suppose that $\rho\in AM(\mathscr{C}_{A_1\times A_2})$ with $\supp(\psi)=K$ as obtained in previous Section~\ref{ex2} and let $\theta_1\in AM(\mathscr{C}_{A_1})$ such that $\supp(\theta_1)=\pi_1(K)$.
Then for $n\ge 1$ and $(w_1,w_2)\in A_1^n\times A_2^n$ let 
$$
\psi ([w_1]\times [w_2])=\theta_1([w_1])\frac{\rho([w_1]\times[w_2])}{I_{\rho,1}([w_1])}.
$$
We have $\psi\in AM(\mathscr{C}_{A_1\times A_2})$ because $\theta_1$, $\rho$ and $I_{\rho,1}$ are almost multiplicative. If we take for $\rho$ a function constructed like in Section~\ref{ex1} and for $\theta_1$ the restriction to $\mathscr{C}_{A_1}$ of an ergodic quasi-Bernoulli measure whose support is equal to $\pi_1(K)$, the measure $\mu_\psi$ constructed from $\psi$ in Section~\ref{AM} is a Gibbs measure like those considered in \cite{barmen} (the interpretation of these measures in terms of the thermodynamic formalism for random transformations \cite{BoGu,KK} is provided in \cite{barmen}). The model proposed in this section is more general.

\section{Main results}
\subsection{{\bf Multifractal analysis on self-affine symbolic spaces}\label{results}}

We fix an element $\psi$ of $AM(\mathscr{C}_{A_1\times A_2})$ and we assume without loss of generality that $P(\log(\psi))=0$ and $\psi$ is the restriction to $\mathscr{C}_{A_1\times A_2}$ of a Gibbs measure (see Section~\ref{AM}). We are interested in finding the Hausdorff spectra associated with the level sets 
$$
E_\psi(\alpha)=\left\{(x,y)\in \mbox{supp}(\psi): \lim_{n\to\infty}\frac{\log_{r_2} \psi([x|n]\times [y|n])}{-n}=\alpha\right \},\quad  (\alpha \ge 0)
$$
and 
\begin{eqnarray*}
E_{\mu_\psi}(\alpha)&=&\left\{(x,y)\in \mbox{supp}(\psi): \lim_{n\to\infty}\frac{\log_{r_2} \mu_\psi\big (B_n(x,y)\big )}{-n}=\alpha\right \},\quad (\alpha\ge 0).
\end{eqnarray*}
These spectra are respectively defined by 
$$
d_\psi:\ \alpha\ge 0\mapsto \dim\, E_\psi(\alpha)\mbox{ and } d_{\mu_\psi}: \ \alpha\ge 0\mapsto  \dim\, E_{\mu_\psi}(\alpha).
$$

Let $\displaystyle s=\frac{\log r_1}{\log r_2}$ and for $f:\mathbb{R}\to\mathbb{R}\cup\{-\infty\}$, define the Legendre transform of $f$ by $f^*:\alpha\ge 0\mapsto \inf_{q\in\mathbb{R}}\alpha q-f(q)$.

\subsubsection{\bf Multifractal analysis of Birkhoff averages.}\label{Th2} For $q\in \mathbb{R}$ define
$$
\mathcal{T}_\psi(q)=-P\big (s\log( I_{\psi,q})\big )/\log (r_1).
$$
It follows from Proposition~\ref{AMcont1}(2) that the function $\mathcal{T}_\psi$ is concave. Also, due to Proposition-Definition~\ref{AM1}, there exists a constant $C$ (depending on $\psi$ and $q$ only) such that 
\begin{equation}\label{submul1}
\sum _{w_1\in A_1^n}I_{\psi,q}([w_1])^s\approx r_1^{-n \mathcal{T}_\psi(q)}, \ n\ge 1.
\end{equation}
\begin{theorem}\label{th2}
(i) The concave function $\mathcal{T}_\psi$ is differentiable and non decreasing. 

\smallskip

(ii) For every $\alpha\in\mathbb{R}_+$, $d_{\psi}(\alpha)=\mathcal{T}_\psi^*(\alpha)$ if $\mathcal{T}_\psi^*(\alpha)> 0$ and $E_{\psi}(\alpha)=\emptyset$ if $\mathcal{T}_\psi^*(\alpha)<0$. 
\end{theorem}

\begin{remark}\label{submul} If in Example \ref{ex1} we assume  that the restriction of $\varphi$ to each set $K\cap [a_1]\times[a_2]$, $(a_1,a_2)\in A$, takes a constant value $\varphi_{a_1,a_2}$ then
$$
\mathcal{T}_\psi(q)=-\log_{r_1}\sum_{a_1\in A_1}\Bigl[\sum_{a_2\in A_2}\exp(q\varphi_{a_1,a_2})\Bigr ]^s,
$$
with the convention $\exp(q\varphi_{a_1,a_2})=0$ if $(a_1,a_2)\not\in A$. In this case, Theorem~\ref{th2} is obtained by seeking, for each $\alpha\ge 0$, a Bernoulli measure of maximal Hausdorff dimension supported by $E_\psi(\alpha)$. Then $\mathcal{T}_\psi(q)$ appears as a Lagrange multiplier (the same approach was used in \cite{King} for the multifractal analysis of Bernoulli measures on Sierpinski carpets).  Our general result is an extension of this situation, and $d_\psi(\alpha)$ is the maximal Hausdorff dimension of an ergodic measure on $\supp(\psi)$ supported by $E_\psi(\alpha)$. Such a measure is introduced in Section~\ref{aux}.
\end{remark}

\subsubsection{\bf An extension of the result obtained in \cite{barmen} for $d_{\mu_\psi}$.}\label{Th1} The central property used in \cite{barmen} was that the restriction to $\mathscr{C}_{A_1\times A_2}$ of the Gibbs measures considered in that paper (see the discussion of Section~\ref{ex3}) satisfy  the same property as $\psi$, namely (\ref{bernoulli}). We obtain the following generalized form of Theorem 1.1 of \cite{barmen}.

\smallskip

For $q\in\mathbb{R}$ define 
$$
\beta_\psi(q)=-P\big (q(1-s)\log (I_{\psi,1})+s\log (I_{\psi,q})\big )/\log(r_1). 
$$
It follows from Proposition~\ref{AMcont1}(2) that the function $\beta_\psi$ is concave. Also, due to Proposition-Definition~\ref{AM1}, there exists a constant $C$ (depending on $\psi$ and $q$ only) such that 
\begin{equation}\label{submul2}
\sum_{w_1\in A_1^n}\ I_{\psi,1}([w_1])^{q(1-s)}I_{\psi,q}([w_1])^s\approx r_1^{-n\beta_\psi(q)}, \ n\ge 1.
\end{equation}

\begin{theorem}\label{th1}
(i) The concave function $\beta_\psi$ is differentiable and non decreasing. 

\smallskip

(ii)  For every $\alpha\in\mathbb{R}_+$,  $d_{\mu_\psi}(\alpha)=\beta_\psi^*(\alpha)$ if $\beta_\psi^*(\alpha)> 0$ and $E_{\mu_\psi}(\alpha)=\emptyset$ if $\beta_\psi^*(\alpha)<0$. 
\end{theorem}

\begin{remark}
In \cite{barmen}, the function $\beta_\psi$ is not interpreted as a topological pressure. 
\end{remark}

\subsection{{\bf Geometric realizations on self-affine Sierpinski carpets}}\label{proj}

Let $\mathbb{T}=\mathbb{R}/\mathbb{Z}$. Let $A$ and $K$ be defined as in Section~\ref{Examples} and let $\widetilde K$ be the self-affine Sierpinski carpet obtained in the torus $\mathbb{T}^2$ as the attractor of the iterated function system 
$$
S_{a_1,a_2}:(x,y)\mapsto (a_1r_1^{-1}+r_1^{-1}x,a_2 r_2^{-1}+r_2^{-1}y), \ \ (a_1,a_2)\in A.$$
By construction $\widetilde K=\widetilde \pi (K)$.

Let $\widetilde A_1=\{a_1\in A_1: \ \exists \ a_2\in A_2,\ (a_1,a_2)\in A\}$, and for $a_1\in A_1$ let $A_1(a_1)=\{a_2\in A_2:\ (a_1,a_2)\in A\}$.

For $i\in\{1,2\}$ define the canonical projections from $\mathbb{A}_i$  to $\mathbb{T}$
$$
\widetilde \pi_i: z\in {\mathbb{A}}_i\mapsto \sum_{k\ge 1} z_kr_i^{-k},  \mbox{ and let } \widetilde \pi=(\widetilde\pi_1,\widetilde\pi_2).
$$

The torus $\mathbb{T}^2$ is endowed with the transformation $T:(x,y)\mapsto (r_1 x,r_2 y)$ and the canonical metric distance denoted $\delta$.

\subsubsection{{\bf Multifractal analysis of Birkhoff averages.}}\label{proj11}
Let $\widetilde \varphi:\widetilde K\subset \mathbb{T}^2 \to \mathbb{R}$ be a continuous function satisfying the same Dini condition as $\varphi$ in (\ref{Dini}) but with $d$  replaced by $\delta$ and $K$ replaced by $\widetilde K$.  The mapping $\varphi=\widetilde \varphi\circ \widetilde \pi$ satisfies (\ref{Dini}), and we associate $\psi$ with $\varphi$ as in Section~\ref{ex1}. By considering $\widetilde\varphi-P(\log (\psi))$ instead of $\widetilde \varphi$ if necessary, we can assume without loss of generality that the topological pressure of $\psi$ is equal to 0.
For $\beta\in\mathbb{R}$ let
$$
F_{\widetilde \varphi}(\beta)=\left \{z\in\widetilde K: \lim_{n\to\infty}\frac{S_n\widetilde\varphi (z)}{n}=\beta\right\}.
$$

\begin{theorem}\label{proj1}
For all $\beta\in\mathbb{R}$, $\dim\, F_{\widetilde \varphi}(\beta)= \mathcal{T}_\psi^*(-\beta/\log r_2)$ if $\mathcal{T}_\psi^*(-\beta/\log r_2)>0$ and $F_{\widetilde \varphi}(\beta)=\emptyset$ if $\mathcal{T}_\psi^*(-\beta/\log r_2 )<0$.
\end{theorem}

\subsubsection{{\bf Multifractal analysis of Gibbs measures.}}

Let $\psi\in AM(\mathscr{C}_{A_1\times A_2})$ such that $\supp (\psi)=K$ and $\psi$ is normalized to be the restriction of a Gibbs measure $\mu_\psi$ to  $\mathscr{C}_{A_1\times A_2}$. Let $\widetilde\mu_\psi=\mu_\psi\circ \widetilde\pi^{-1}$ be the projection of $\mu_\psi$ on the Sierpinski carpet $\widetilde K$. We now focus on the Hausdorff dimension of the sets 
$$
E_{\widetilde\mu_\psi}(\alpha)=\left\{z\in \widetilde K: \ \lim_{r\to 0^+}\frac{\log \widetilde\mu_\psi(B(x,r))}{\log (r)}=\alpha\right\},\ \alpha\ge 0.
$$
Theorem 1.2 of \cite{barmen} deals with a special case of the construction of $\psi$ proposed in Section~\ref{ex3}. It possesses the following extension. 

We need to introduce three properties, namely {\bf (P1)}, {\bf (P2)} and {\bf (P3)}. Properties {\bf (P1)} and{\bf (P2)} correspond to properties {\bf (G1)} and {\bf (G2)} of \cite{barmen}, and property {\bf (P3)} is the extension of property {\bf (G3)} of \cite{barmen} to the more general situation studied in this paper.

\smallskip

 {\bf (P1)} $|a_1-a_1'|\ge 2$ for every pair $(a_1,a_1')$ of distinct elements of $\widetilde A_1$. 

\smallskip

{\bf (P2)}  $\{0,r_1-1\}\cap (A_1\setminus \widetilde A_1)\neq\emptyset $.

\smallskip

{\bf (P3)} $\{0,r_1-1\}\subset \widetilde A_1$ and for all $q>0$, 
$$
\lim_{n\to\infty} \frac{1}{n}\log I_{\psi,q}([0^{\cdot n}])=\lim_{n\to\infty} \frac{1}{n}\log I_{\psi,q}([(r_1-1)^{\cdot n}]),
$$
where for $j\in A_i$ and $n\ge 1$, $j^{\cdot n}$ stands for the word of length $n$ whose letters are all equal to $j$.

\begin{theorem}\label{proj2}
(i) For every $\alpha\in\mathbb{R}_+$ such that  $\beta_\psi^*(\alpha)> 0$, $\dim\, E_{\widetilde \mu_\psi}(\alpha)\ge \beta_\psi^*(\alpha)$.

\smallskip

(ii)  If $\alpha\ge \beta_\psi'(0)$ then  $\dim\, E_{\widetilde \mu_\psi}(\alpha)\le \beta_\psi^*(\alpha)$ and $E_{\widetilde \mu_\psi}(\alpha)=\emptyset$ if $\beta_\psi^*(\alpha)<0$. 

\smallskip

(iii) Suppose that one of the properties {\bf (P1)}, {\bf (P2)} or {\bf (P3)} holds. 

If $0\le \alpha <\beta_\psi'(0)$ then $\dim\, E_{\widetilde \mu_\psi}(\alpha)\le \beta_\psi^*(\alpha)$ and $E_{\widetilde \mu_\psi}(\alpha)=\emptyset$ if $\beta_\psi^*(\alpha)<0$. 
\end{theorem}


\section{Proofs of Theorems~\ref{th2} and \ref{th1}}\label{proof1}

\subsection{Auxiliary measures}\label{aux}$\ $

Recall that we have adopted in Section~\ref{properties} the convention $0^q=0$ for all $q\in\mathbb{R}$. For $q\in\mathbb{R}$, $n\ge 1$ and $(w_1,w_2)\in A_1^n\times A_2^n$ let 
$$
\begin{cases}
\displaystyle \theta_q([w_1])=r_1^{n\beta_\psi(q)}I_{\psi,1}([w_1])^{q(1-s)}I_{\psi,q}([w_1])^s \\
\displaystyle\widetilde \theta_q([w_1])=r_1^{n\mathcal{T}_\psi(q)} I_{\psi,q}([w_1])^s 
\end{cases}
$$
and 
$$
\begin{cases}
\displaystyle\psi_q([w_1]\times [w_2])=\theta_q([w_1])\frac{ \psi([w_1]\times [w_2])^q}{I_{\psi,q}([w_1])}\\
\displaystyle\widetilde\psi_q([w_1]\times [w_2])=\widetilde \theta_q([w_1])\frac{ \psi([w_1]\times [w_2])^q}{I_{\psi,q}([w_1])}
\end{cases}.
$$
It follows from Proposition~\ref{AMcont1}(1) that the functions $\theta_q$ and $\widetilde\theta_q$ belong to $AM(\mathscr{C}_{A_1})$. Consequently, the functions $\psi_q$ and $\widetilde\psi_q$ belong to $AM(\mathscr{C}_{A_1\times A_2})$. Moreover, due to (\ref{submul1}) and (\ref{submul2}), by construction the topological pressures of $\log \psi_q$ and $\log \widetilde \psi_q$ are equal to 0. Thus, the discussion of Section~\ref{AM} yields two ergodic measures $\mu_{\psi_q}$ and $\mu_{\widetilde \psi_q}$ on $\supp(\psi)$ that are respectively strongly equivalent to $\psi_q$ and $\widetilde \psi_q$ over $\mathscr{C}_{A_1\times A_2}$. In the sequel we denote these measures by $\mu_q$ and $\widetilde \mu_q$ respectively,  and the measure $\mu_\psi$ by $\mu$.  Notice that $\mu_1=\mu$.

\smallskip

The following proposition exhibits the fundamental relation that links the $\mu_q$-measure and the $\mu$-measure of a ball, as well as the fundamental relation that links the $\widetilde \mu_q$-measure of a ball and the $\mu$-measure of the cylinder of same generation the ball is contained in. 

\begin{proposition}\label{prop1}
Let $q\in\mathbb{R}$. For all $(x,y)\in \supp(\psi)$ and $n\ge 1$, define 
$$
u_n(x)=\displaystyle \frac{I_{\psi,q}([x|n])}{I_{\psi,1}([x|n])^q}\text{ and }\widetilde u_n(x)=I_{\psi,q}([x|n]). 
$$
There exists a constant $C$ depending on $\psi$ and $q$ only such that 
\begin{enumerate}
\item 
$\displaystyle 
\frac{\mu_q\big (B_n(x,y)\big )}{\mu\big (B_n(x,y)\big )^qr_2^{n\beta_\psi(q)} }\approx \frac{u_n(x)}{u_{g(n)}(x)^s}, \ (x,y)\in \supp(\psi),\ n\ge 1; $
\item
$\displaystyle
 \frac{\widetilde \mu_q\big (B_n(x,y)\big )}{\psi([x|n]\times [y|n])^q r_2^{n\mathcal{T}_\psi(q)}}\approx \frac{\widetilde u_n(x)}{\widetilde u_{g(n)}(x)^s}, \ (x,y)\in \supp(\psi),\ n\ge 1.$
\end{enumerate}
\end{proposition}

\begin{proof}
For $(x,y)\in \supp(\psi)$ and $n\ge 1$, the ball $B_n(x,y)=[x|g(n)]\times [y|n]$ has the form $[w_1\cdot\widetilde{w}_1]\times
[w_2]$, where $(w_1,\widetilde w_1)\in A_1^n\times A_1^{g(n)-n}$ and $w_2\in A_2^n$. 
\smallskip

(1) Due to Proposition~\ref{AMcont} and the fact that $\theta_q$ is strongly equivalent to $\mu_q\circ \pi_1^{-1}$ over $\mathscr{C}_{A_1}$, by using the definitions and properties of $\psi_q$, $\theta_q$ and $\theta_1$ we have\begin{eqnarray*}
\frac{\mu_q\big (B_n(x,y)\big )}{\mu\big (B_n(x,y)\big )^q}&\approx&\frac{\psi_q([w_1]\times [w_2])\theta_q([\widetilde w_1])}{\psi([w_1]\times [w_2])^q \theta_1([\widetilde w_1])^q}\\
&=&\frac{\theta_q([w_1])\theta_q(\widetilde w_1)I_{\psi,1}([w_1])^q}{\theta_1([w_1])^q\theta_1([\widetilde w_1])^q I_{\psi,q}([w_1])}\\
&\approx& \frac{\theta_q([w_1\cdot \widetilde w_1])}{\theta_1([w_1\cdot \widetilde w_1])^q}\frac{I_{\psi,1}([w_1])^q}{I_{\psi,q}([w_1])}\\
&\approx& \frac{r_1^{g(n)\beta_\psi(q)}I_{\psi,1}([w_1\cdot \widetilde w_1])^{q(1-s)}I_{\psi,q}([w_1\cdot \widetilde w_1])^s}{r_1^{g(n)\beta_\psi(1)}I_{\psi,1}([w_1\cdot \widetilde w_1])^q} \frac{I_{\psi,1}([w_1])^q}{I_{\psi,q}([w_1])}\\
&\approx& r_2^{n\beta_\psi(q)} \frac{I_{\psi,q}([x|n])}{I_{\psi,1}([x|n])^q}\left[\frac{I_{\psi,1}([x|g(n)])^{q}}{I_{\psi,q}([x|g(n)])}\right ]^s.
\end{eqnarray*}

(2) We use the strong equivalence of $\widetilde \mu_q\circ \pi_1^{-1}$ and $\widetilde \theta_q$ and proceed as for (1):

\begin{eqnarray*}
\widetilde\mu_q(B_n(x,y))&\approx &\widetilde\theta_q([\widetilde{w}_1])\widetilde \mu_q([w_1]\times
[w_2])\\
&\approx &
\widetilde\theta_q([\widetilde{w}_1])\widetilde\theta_q([{w}_1])\frac{\psi([w_1]\times[w_2])^q}{I_{\psi,q}([w_1])}\\
&\approx & \widetilde\theta_q([w_1\cdot\widetilde w_1])\frac{\psi([w_1]\times[w_2])^q}{I_{\psi,q}([w_1])}\\
&\approx& \frac{\psi([w_1]\times[w_2])^q}{\sum_{w_1'\in A_1^{g(n)}}I_{\psi,1}(w_1')^s}\frac{I_{\psi,q}([w_1\cdot\widetilde w_1])^s}{I_{\psi,q}([w_1])}\\
&\approx &r_1^{g(n)\mathcal{T}_\psi(q)}\psi([w_1]\times[w_2])^q\frac{I_{\psi,q}([w_1\cdot\widetilde w_1])^s}{I_{\psi,q}([w_1])}.
\end{eqnarray*}
\end{proof}

\subsection{The differentiability of $\beta_\psi$ and $\mathcal{T}_\psi$}$\ $

The $L^q$-spectrum $\tau_\mu$ of $\mu$ is equal to
$$
\tau_\mu (q)=\liminf_{n\to\infty}\tau_{\mu,n}(q), \quad\mbox{where }\tau_{\mu,n}(q)=-\frac{1}{n}\log_{r_2}\sum_{B\in \mathcal{F}_n}\mu(B)^q
$$
still with the convention $0^q=0$ (recall that $\mathcal{F}_n$ is defined in (\ref{Fn})).

\begin{proposition}\label{diff}
The functions $\beta_\psi$ and $\tau_\mu$ are differentiable at 1 and $\beta_\psi'(1)=\tau_\mu'(1)$. In particular, the measure $\mu$ is monodimensional in the following sense: For $\mu$-almost every $z\in \supp(\mu)$, $
\displaystyle \lim_{n\to\infty}\frac{\log\mu\big (B_n(z)\big )}{\log |B_n(z)|}=\beta_\psi'(1).
$ Consequently, if $B$ is a Borel set of positive $\mu$-measure, then $\dim\, B\ge \beta_\psi'(1)$. 
\end{proposition}
\begin{proof}
The result regarding the differentiability of $\beta_\psi$ and $\tau_\mu$ at $1$ is proved in the same way as Proposition 2.3 in \cite{barmen} (notice that in \cite{barmen} which deals with a smaller class of Gibbs measures, $\mu\circ\pi_1^{-1}$ is denoted by $\mathbb{P}$). The result regarding the monodimensionality  comes from the existence of $\tau_\mu'(1)$ (see~\cite{Ngai} or \cite{He}). The last property is a direct consequence of the monodimensionality and the mass distribution principle (see pp 136--145 of \cite{Bil}, Section 4.2 of \cite{Falcbook} or p. 43 of \cite{Pesin}).
\end{proof}

Part (1) of the next proposition is stated without proof in \cite{barmen} for a special class of measures in $AM(\mathscr{C}_{A_1\times A_2})$. However the property it claims is not so easy to check and we are going to prove it. Part (2) of this proposition provides the new fundamental relation that links the function $\mathcal{T}_\psi$ to the family of functions  $\left \{\beta_{\widetilde \psi_q}\right\}_{q\in\mathbb{R}}$. 
\begin{proposition}\label{prop4}
For all $(q,r)\in\mathbb{R}$ we have:
\begin{enumerate}

\item $\mathcal{\beta}_{ \psi_q}(r)=\beta_\psi(qr)-r\beta_\psi(q)$.

\item $\mathcal{\beta}_{\widetilde \psi_q}(r)=\mathcal{T}_\psi(qr)-r\mathcal{T}_\psi(q)$.
\end{enumerate}
\end{proposition}

\begin{proof} Let $(q,r)\in\mathbb{R}$.

(1) By (\ref{submul2}) for $n\ge 1$ we have 
\begin{eqnarray*}
r_1^{-n\beta_{\psi_q}(r)} &\approx&  \sum_{w_1\in A_1^n}I_{\psi_q,1}([w_1])^{r(1-s)}I_{\psi_q,r}([w_1])^s\\\
&\approx& \sum_{w_1\in A_1^n}\ \Big [\sum_{w_2\in A_2^n}\psi_q([w_1]\times [w_2])\Big ]
^{r(1-s)}\Big [\sum_{w_2\in A_2^n}\psi_q([w_1]\times [w_2])^r\Big ]^s.
\end{eqnarray*}
Also, by definition we have
$$
\sum_{w_2\in A_2^n}\psi_q([w_1]\times [w_2])=\theta_q([w_1])=r_1^{n\beta_\psi(q)}I_{\psi,1}([w_1])^{q(1-s)}I_{\psi,q}([w_1])^s 
$$ 
as well as
\begin{eqnarray*}
\sum_{w_2\in A_2^n}\psi_q([w_1]\times [w_2])^r&=&\frac{\theta_q([w_1])^r}{I_{\psi,q}([w_1]^r}\sum_{w_2\in A_2^n}\psi([w_1]\times [w_2])^{qr}\\ 
&=& r_1^{nr\beta_\psi(q)}I_{\psi,1}([w_1])^{qr(1-s)}I_{\psi,q}([w_1])^{r(s-1)} I_{\psi,qr}([w_1]) 
.
\end{eqnarray*}
Inserting the two previous equalities in the approximation of $r_1^{-n\beta_{\psi_q}(r)}$ and using (\ref{submul2}) again we find that the powers of $I_{\psi,q}([w_1])$ vanish and 
\begin{eqnarray*}
r_1^{-n\beta_{\psi_q}(r)} \approx  r_1^{nr\beta_\psi(q)} \sum_{w_1\in A_1^n}I_{\psi,1}([w_1])^{qr(1-s)} I_{\psi,qr}([w_1])^s\approx r_1^{nr\beta_\psi(q)} r_1^{-n\beta_\psi(qr)} .
\end{eqnarray*}

\smallskip

(2) By (\ref{submul1}) for $n\ge 1$ we have 
\begin{eqnarray*}
r_1^{-n\beta_{\widetilde \psi_q}(r)} &\approx&  \sum_{w_1\in A_1^n}I_{\widetilde \psi_q,1}([w_1])^{r(1-s)}I_{\widetilde \psi_q,r}([w_1])^s\\\
&\approx& \sum_{w_1\in A_1^n}\ \Big [\sum_{w_2\in A_2^n}\widetilde \psi_q([w_1]\times [w_2])\Big ]
^{r(1-s)}\Big [\sum_{w_2\in A_2^n}\widetilde \psi_q([w_1]\times [w_2])^r\Big ]^s.
\end{eqnarray*}
Also, $\displaystyle \sum_{w_2\in A_2^n}\widetilde \psi_q([w_1]\times [w_2])=\widetilde \theta_q([w_1])=r_1^{n\mathcal{T}_\psi(q)}I_{\psi,q}([w_1])^s $ and
\begin{eqnarray*}
\sum_{w_2\in A_2^n}\widetilde \psi_q([w_1]\times [w_2])^r&=&\frac{\widetilde \theta_q([w_1])^r}{I_{\psi,q}([w_1]^r}\sum_{w_2\in A_2^n}\psi([w_1]\times [w_2])^{qr}\\ 
&=&r_1^{nr\mathcal{T}_\psi(q)}I_{\psi,q}([w_1])^{r(s-1)} I_{\psi,qr}([w_1]) 
.
\end{eqnarray*}
Inserting the two previous equalities in the approximation of $r_1^{-n\beta_{\widetilde \psi_q}(r)}$ and using (\ref{submul1}) again yields
$\displaystyle r_1^{-n\beta_{\widetilde \psi_q}(r)} \approx  r_1^{nr\mathcal{T}_\psi(q)} \sum_{w_1\in A_1^n}I_{\psi,qr}([w_1])^s\approx r_1^{nr\mathcal{T}_\psi(q)} r_1^{-n\mathcal{T}_\psi(qr)}$.
\end{proof}

\begin{corollary}\label{coro1}
The functions $\beta_\psi$ and $\mathcal{T}_\psi$ are differentiable.
\end{corollary}
\begin{proof}
Proposition~\ref{diff} implies that $\beta_{\psi_q}$ and $\beta_{\widetilde \psi_q}$ are differentiable at 1, so that parts (1) and (2) of Proposition~\ref{prop4} yield directly the differentiability of $\beta_\psi$ and $\mathcal{T}_\psi$ at $q\neq 0$ with the relations $\beta_{\psi_q}'(1)=q\beta_\psi'(q)-\beta_\psi(q)$ and $\beta_{\widetilde\psi_q}'(1)=q\mathcal{T}_\psi'(q)-\mathcal{T}_\psi (q)$. For the case $q=0$, we mimic the approach of \cite{He} for the $L^q$-spectrum of a quasi-Bernoulli measure. Property (\ref{submul1}) can be made precise as follows: There exists $C>0$ such that for all $n,m\ge 1$ and $q\in\mathbb{R}$ we have $
-C|q|\le (m+n)\mathcal{T}_{\psi,m+n}(q)-m\mathcal{T}_{\psi,m}(q)-n\mathcal{T}_{\psi,n}(q)\le C|q|$.
This yields $|\mathcal{T}_{\psi}(q)-\mathcal{T}_{\psi,n}(q)|\le C|q|/n$ for all $n\ge 1$ and $q\in\mathbb{R}$. Since $\mathcal{T}_{\psi,n}(0)=\mathcal{T}_{\psi}(0)$ for all $n\ge 1$, we have for all $n\ge 1$ and $q\in\mathbb{R}^*$, $\left |\frac{\mathcal{T}_{\psi}(q)-\mathcal{T}_{\psi}(0)}{q}-\frac{\mathcal{T}_{\psi,n}(q)-\mathcal{T}_{\psi,n}(0)}{q}\right |\le C/n$. Consequently, since $\mathcal{T}_\psi$ is concave and $\mathcal{T}_{\psi_n}'(0)$ exists for all $n\ge 1$, the previous inequality implies that the left and right derivatives of $\mathcal{T}_\psi$ at 0 are equal.  The same argument works for $\beta_\psi$.
\end{proof}

\subsection{Upper bounds for the dimensions} $\ $

Recall Proposition~\ref{prop1}.

\begin{proposition}\label{limsup}
Let $q\in\mathbb{R}$. For all $x\in \pi_1(\supp(\psi))$ we have 
$$
\min \left  (\limsup_{n\to\infty}\left (\frac{u_n(x)}{u_{g(n)}(x)^s}\right )^{1/n}, \limsup_{n\to\infty}\left (\frac{\widetilde u_n(x)}{\widetilde u_{g(n)}(x)^s}\right )^{1/n}\right )\ge 1.
$$
\end{proposition}

\begin{proof}
Let $v\in \{u,\widetilde u\}$ and $x\in \pi_1(\supp(\psi))$. We use an argument introduced in \cite{McMullen}. We write
$$
\left (\frac{v_n(x)}{v_{g(n)}(x)^s}\right )^{1/n}=\frac{v_n(x)^{1/n}}{v_{g(n)}(x)^{1/g(n)}}
v_{g(n)}(x)^{1/g(n)-s/n}.
$$
Now, since  by construction there exists $0<a<b<\infty$ independent of $x$ such that $a^n\le v_n(x)\le b^n$ for all $n\ge 1$, and $|n/g(n)-s |=O(1/n)$, we have $\lim_{n\to\infty} v_{g(n)}(x)^{1/g(n)-s/n}=1$. The conclusion then comes from the fact that since $s=\lim_{n\to\infty} n/g(n)<1$, we must have $\limsup_{n\to\infty}\frac{v_n(x)^{1/n}}{v_{g(n)}(x)^{1/g(n)}}\ge 1$ because $v_n(x)^{1/n}$ is bounded away from 0. 
\end{proof}

\begin{corollary}\label{corupper}
For all $\alpha\ge 0$, we have $\dim\, E_{\mu}(\alpha)\le \beta_\psi^*(\alpha)$ and $\dim\, E_{\psi}(\alpha)\le \mathcal{T}_\psi^*(\alpha)$, a negative dimension meaning that the set is empty.
\end{corollary}

\begin{proof}
We need the following lemma. It is a simple adaptation to $\mathbb{A}_1\times \mathbb{A}_2$ of Proposition 4.9(b) of  \cite{Falcbook} which enables to find upper bounds for the Hausdorff dimension of  sets in $\mathbb{R}^n$. 
\begin{lemma}\label{falc}
Let $F$  be non-empty subset of $\mathbb{A}_1\times \mathbb{A}_2$. Suppose that there exists a positive Borel measure $\nu$ on $\mathbb{A}_1\times \mathbb{A}_2$ and $\gamma\ge 0$ such that $\displaystyle \liminf_{n\to\infty}\frac{\log\nu(B_n(z))}{\log |B_n(z)|}\le \gamma$ for all $z\in F$. Then $\dim\, F\le \gamma$. 
\end{lemma}
Let $\alpha\in\mathbb{R}$ and $q\in\mathbb{R}$. It follows from Proposition~\ref{limsup} that if $E_\mu(\alpha)$ is not empty, then for all $z\in E_\mu(\alpha)$ we have $\displaystyle \liminf_{n\to\infty}\frac{\log_{r_2}\mu_q(B_n(z))}{-n}\le \alpha q-\beta_\psi(q)$. Then Lemma~\ref{falc} yields $\dim\, E_\mu(\alpha)\le \alpha q-\beta_\psi(q)$. Since this holds for all $q\in\mathbb{R}$, we have the desired conclusion. Similarly, if $E_\psi(\alpha)$ is not empty, then for all $q\in\mathbb{R}$ and $z\in E_\psi(\alpha)$ we have $\displaystyle \liminf_{n\to\infty}\frac{\log_{r_2}\widetilde\mu_q(B_n(z))}{-n}\le \alpha q-\mathcal{T}_\psi(q)$.
\end{proof}

\noindent
{\it Proof of Lemma~\ref{falc}.} Fix $\varepsilon>0$. Let $\delta>0$ and for all $z\in F$ let $n_z\ge 1$ such that the diameter $r_2^{-n_z}$ of $B_{n_z}(z)$  is less than $\delta$ and $|B_{n_z}(z)|^{\gamma+\varepsilon}\le \nu (B_{n_z}(z))$. Let $\mathcal{F}=\{ B_{n_z}(z):\ z\in F\}$. The elements of $\mathcal{F}$ form a $\delta$-covering of $F$ and 
$$
\sum_{B\in\mathcal{F}}|B|^{\gamma+2\varepsilon}\le \sum_{B\in\mathcal{F}}r_2^{-n_z\varepsilon}\nu (B_{n_z}(z))\le \sum_{n\ge 1}r_2^{-n}\sum_{B\in\mathcal{F}_n}\nu (B)\le \sum_{n\ge 1}r_2^{-n}\|\nu\|<\infty.$$
It results from this finite bound independent of $\delta$ that $\dim\, F\le \gamma+2\varepsilon$. Since this holds for all $\varepsilon>0$, we have the desired conclusion. 
 
\subsection{Lower bounds for the dimensions}

\begin{proposition}\label{prop5}
Let $q\in\mathbb{R}$. We have 
$$
\begin{cases}
\mu_q\big (E_{\mu_q}(\beta_\psi^*(\beta_\psi'(q)))\big )=\mu_q\big (E_{\mu}(\beta_\psi'(q))\big )=1,\\
\widetilde \mu_q\big (E_{\widetilde \mu_q}( \mathcal{T}_\psi^*(\mathcal{T}_\psi'(q)))\big )=\widetilde\mu_q\big (E_{\psi}(\mathcal{T}_\psi'(q))\big )=1.
\end{cases}
$$
Consequently, $\dim\, E_{\mu}(\beta_\psi'(q))\ge \beta_\psi^*(\beta_\psi'(q))$ and $\dim\, E_{\psi}(\mathcal{T}_\psi'(q))\ge \mathcal{T}_\psi^*(\mathcal{T}_\psi'(q))$.
\end{proposition}

\begin{proof}
Suppose first that $q\neq 0$. Due to Proposition~\ref{prop4} and its Corollary~\ref{coro1} we know that $\beta_{\psi_q}'(1)=q\beta_\psi'(q)-\beta_\psi(q)=\beta_\psi^*(\beta_\psi'(q))$ and $\beta_{\widetilde \psi_q}'(1)=q\mathcal{T}_\psi'(q)-\mathcal{T}_\psi(q)= \mathcal{T}_\psi^*(\mathcal{T}_\psi'(q))$. Thus, Proposition~\ref{diff} yields $\mu_q\big (E_{\mu_q}(q\beta_\psi'(q)-\beta_\psi(q))\big )=1$ and $\widetilde \mu_q\big (E_{\widetilde \mu_q}(q\mathcal{T}_\psi'(q)-\mathcal{T}_\psi(q))\big )=1$. 


Now, for $x\in \pi_1(\supp(\psi))$ let $(u_n(x))_{n\ge 1}$ and $(\widetilde u_n(x))_{n\ge 1}$ be defined as in Proposition~\ref{prop1}. Proposition~\ref{AMcont1}(1) implies $u_{n+p}(x)\approx u_n(x)u_p(\sigma_1^n(x))$ for all $n,p\ge 1$ and $x\in\pi_1(\supp(\psi))$. The same property holds for $(\widetilde u_n(\cdot))_{n\ge 1}$. Consequently, $\mu_q$ and $\widetilde\mu_q$ being ergodic, Kingman's submultiplicative ergodic theorem (\cite{K}) applied to $u_n\circ \pi_1$ and $\mu_q$ and then to $\widetilde u_n\circ \pi_1$ and $\widetilde \mu_q$ yields $l_q$ and $\widetilde l_q$ in $\mathbb{R}$ such that $\lim_{n\to\infty}\frac{\log u_n(x)}{n}=l_q$, $\mu_q$ almost-everywhere (a.\,e.), and $\lim_{n\to\infty}\frac{\log \widetilde u_n(x)}{n}=\widetilde l_q$, $\widetilde \mu_q$-a.\,e.. Since $\lim_{n\to\infty} n/g(n)=s$, this implies $\lim_{n\to\infty}\frac{1}{n}\log \left(\frac{u_n(x)}{u_{g(n)}(x)^s}\right )=0$, $\mu_q$-a.\,e., and $\lim_{n\to\infty}\frac{1}{n}\log \left(\frac{\widetilde u_n(x)}{\widetilde u_{g(n)}(x)^s}\right )=0$, $\widetilde \mu_q$-a.\,e.. 

Then, it follows from Proposition~\ref{prop1} that for $\mu_q$-a.\,e. $(x,y)\in\supp(\psi)$, since $(x,y)\in E_{\mu_q}(q\beta_\psi'(q)-\beta_\psi(q))$ and $q\neq 0$, we have $(x,y)\in E_{\mu}(\beta_\psi'(q))$; similarly since  $(x,y)\in E_{\widetilde \mu_q}(q\mathcal{T}_\psi'(q)-\mathcal{T}_\psi(q))$, we have $(x,y)\in E_{\psi}(\mathcal{T}_\psi'(q))$. Thus $\mu_q(E_{\mu}(\beta_\psi'(q))=1$ and $\widetilde\mu_q(E_{\psi}(\mathcal{T}_\psi'(q))=1$. The lower bounds for $\dim\, E_{\mu}(\beta_\psi'(q))$ and $\dim\, E_{\psi}(\mathcal{T}_\psi'(q))$ come now from Proposition~\ref{diff}.

\medskip

If $q=0$, we first notice that $\beta_\psi (0)=\mathcal{T}_\psi(0)$. Also $\mu_0=\widetilde\mu_0$ and this measure, which belongs to the class of multinomial measures, is the measure considered in \cite{McMullen} to find a sharp lower bound for $\dim\, \supp(\psi)$ (\cite{McMullen} deals with the self-affine Sierpinski carpet but its result has a direct transposition to the self-affine symbolic space). It turns out that $\displaystyle \dim\, \supp(\mu_0)=\dim\, \supp(\psi)=-\beta_\psi (0)=-\mathcal{T}_\psi(0)=\log_{r_1}\sum_{a_1\in A_1}( \#\{a_2\in A_2:\ [a_1]\times[a_2]\cap \supp(\psi)\neq\emptyset\})^s$. Moreover, if $B$ is a Borel set of positive $\mu_0$-measure, then $\dim\, B\ge -\beta_\psi (0)$. 

Let us suppose for a while that there exists $(\alpha,\widetilde \alpha)\in\mathbb{R}_+^2$ such that the sets $E_{\mu}(\alpha)$ and $E_\psi(\widetilde \alpha)$ are of full $\mu_0$-measure. We then necessarily  have $\dim\, E_{\mu}(\alpha)=\dim\, E_\psi(\widetilde \alpha)= -\beta_\psi (0)$. Moreover, it follows from Corollary~\ref{corupper} that if $\alpha\neq \beta_\psi'(0)$ then $\dim\, E_\mu(\alpha)<-\beta_\psi (0)$, and if $\widetilde \alpha \neq \mathcal{T}_\psi'(0)$ then $\dim\, E_\psi(\alpha)<-\mathcal{T}_\psi(0)$. Indeed, $\beta_\psi^*$ and $\mathcal{T}_\psi^*$ reach their absolute maximum at $\beta_\psi'(0)$ and $\mathcal{T}_\psi'(0)$ respectively. Consequently, $\alpha=\beta_\psi'(0)$ and $\widetilde\alpha=\mathcal{T}_\psi'(0)$. 

The existence of $\alpha$ and $\widetilde \alpha$ is obtained as follows. By Proposition~\ref{AMcont}(2), there exists $C>0$ such that if $(x,y)\in \supp(\psi)$ and $n\ge 1$ we have 
$$
\mu(B_n(x,y))\approx  \psi([x|n]\times[y|n]) \frac{\theta_1([x|g(n)])}{\theta_1([x|n])}.
$$
Moreover, since $\psi\in AM(\mathscr{C}_{A_1\times A_2})$ and $\theta_1\in AM(\mathscr{C}_{A_1})$, we have $\psi([x|n+p]\times[y|n+p])\approx \psi([x|n]\times[y|n])\psi([\sigma_1^n(x)|p]\times[\sigma_2(y)|p])$ and $\theta_1([x|n+p])\approx \theta_1([x|n])\theta_1([\sigma_1^n(x)|p])$ for all $n,p\ge 1$ and $(x,y)\in\supp(\psi)$. Hence, Kingman's ergodic theorem applied to $(x,y)\mapsto\psi([x|n]\times [y|n])$, $(x,y)\mapsto \theta_1([x|n])$ and the ergodic measure $\mu_0$ yields the conclusion.
\end{proof}

\section{Proofs of Theorems~\ref{proj1} and \ref{proj2}}\label{proof2}

\noindent
{\it Proof of Theorem~\ref{proj1}.}  Let $\widetilde \varphi$, $\varphi$ and $\psi$ be as in Section~\ref{proj11}. Let us make the following observations. On the one hand the projection of a ball of $\mathbb{A}_1\times\mathbb{A}_2$ in the torus is an almost square of comparable diameter, and the family of these projections can be used to cover sets when estimating their Hausdorff dimension. On the other hand the projection of a Gibbs measure $\mu$ on $\mathbb{T}$ does not affect a positive mass to the boundary of such almost squares (this follows for instance from Proposition 4.1 in \cite{barmen}). Consequently, the $\mu$-mass of a ball is equal to the $\mu \circ\widetilde\pi^{-1}$-mass of its projection. Another point is that for $z\in K$, $S_n\varphi(z)/n$ or equivalently $\log \psi ([z|n])/n$ possesses a limit as $n\to\infty$ if and only if it is the case for $S_n\widetilde \varphi (\widetilde \pi (z))$.  Finally, the measures $\widetilde \mu_q\circ\widetilde\pi^{-1}$ can be used to transpose to the sets $E_{\widetilde \varphi}(\beta)$ the results of Corollary~\ref{corupper} and Proposition~\ref{prop5} concerning the sets $E_\psi(\alpha)$. 

\medskip

\noindent
{\it Proof of Theorem~\ref{proj2}.} We said in the introduction that the same arguments as those used in the proof of Theorem 1.2 of \cite{barmen} work in the more general situation studied in this paper. These arguments are too long to be sketched here but they are exposed in detail in~\cite{barmen}.

\medskip


\end{document}